\documentclass[12pt,a4paper]{article}
\usepackage{amssymb}

\usepackage{graphicx}
\usepackage{amsmath}
\usepackage{amsmath,amssymb}%

\usepackage{amsfonts}
\usepackage{amsthm,amscd}
\usepackage[english]{babel}
\usepackage{graphicx}
\usepackage{amsmath}
\usepackage{amssymb}
\usepackage{amstext}

\newtheorem{theorem}{Theorem}
\newtheorem{corollary}[theorem]{Corollary}
\newtheorem{definition}[theorem]{Definition}
\newtheorem{example}[theorem]{Example}
\newtheorem{lemma}[theorem]{Lemma}

\newtheorem{proposition}[theorem]{Proposition}
\newtheorem{remark}[theorem]{Remark}

\title{Duality Theorems in Ergodic Transport}

\author{Artur O. Lopes\thanks{arturoscar.lopes@gmail.com, Instituto de Matem\'atica - UFRGS - Partially supported by DynEurBraz, CNPq, PRONEX -- Sistemas
Dinamicos,  INCT, Convenio Brasil-Franca}\, and Jairo K. Mengue\thanks{jairokras@gmail.com,  Instituto de Matem\'atica - UFRGS} }

\begin{document}

\maketitle

\begin{abstract}

We analyze several problems of Optimal Transport Theory in the setting of Ergodic Theory.
In a certain  class of problems we consider questions in Ergodic Transport which are generalizations of the ones in Ergodic Optimization.

 Another class of problems is the following: suppose $\sigma$ is the shift  acting on Bernoulli space $X=\{0,1\}^\mathbb{N}$,  and, consider a fixed continuous cost function $c:X \times X\to \mathbb{R}$. Denote by  $\Pi$ the set of all Borel probabilities $\pi$ on $X\times X$, such that, both its
$x$ and $y$ marginal are $\sigma$-invariant probabilities. We are interested in the optimal plan $\pi$ which minimizes $\int\, c \,\,d \pi$ among the probabilities on $\Pi$.

We show, among other things, the analogous Kantorovich Duality Theorem. We also analyze uniqueness of the optimal plan under generic assumptions on $c$. We investigate the existence of a dual pair of Lipschitz functions which realizes the present dual Kantorovich problem under the assumption that the cost is Lipschitz continuous. For continuous costs $c$ the corresponding results in the Classical Transport Theory and in Ergodic Transport Theory can be, eventually, different.

We also consider the problem of approximating the optimal plan $\pi$ by convex combinations of plans such that the support projects in periodic orbits.


\end{abstract}



\section{Introduction}

For a compact metric space $X$, we denote $P(X)$ the set of probabilities acting on the Borel sigma-algebra $\mathcal{B}(X)$. $C(X)$ denotes the set of continuous functions on $X$ taking real values.

We denote by $\sigma$ the shift acting on $\{1,2,..,d\}^\mathbb{N}$ and  by $\hat{\sigma}$ the shift acting on $\{1,2,..,d\}^\mathbb{Z}$. Some of our results apply to more general cases where one can consider a continuous transformation defined on any compact metric space $X$. Anyway, the reader can take $\{1,2,..,d\}^\mathbb{Z}= X \times Y=\{1,2,..,d\}^\mathbb{N}\times \{1,2,..,d\}^\mathbb{N}$ as our favorite toy model.

We consider a continuous cost function $c:X\times Y \to \mathbb{R}$, where $X$ and $Y$ are compact metric spaces.

The Classical Transport Problem consider probabilities $\pi$ on $P(X\times Y)$ and the minimization of $\int \, c(x,y) d \pi(x,y)$ under the hypothesis that
the $y$-marginal of $\pi$ is a fixed probability $\nu$ and $x$-marginal of $\pi$ is a fixed probability $\mu$. A probability $\pi$ which minimizes such integral is called an optimal plan \cite{Vi1} \cite{Gi}.

We want to analyze a different class of problem where in some way the restriction to invariant probabilities \cite{Man} appears in some form. We present several different settings.

As a motivation one can ask: given the $2$-Wasserstein metric $W$ on the space of probabilities on $X=\{1,2,..,d\}^\mathbb{N} $, and a certain fixed probability $\mu$, which is not invariant for the shift $\sigma:\{1,2,..,d\}^\mathbb{N}\to \{1,2,..,d\}^\mathbb{N}$, characterize  the closest $\sigma$-invariant probability $\nu$ to $\mu$.  In other words, we can be interested in finding a $\sigma$-invariant probability $\nu$ which minimizes the value $W(\mu,\nu)$ for a fixed $\mu$. In this case we are taking $X=Y$. What can be said about optimal transport plans, duality, etc?

As a generalization of this problem one can consider a continuous cost $c(x,y)$, where $c:\{1,2,..,d\}^\mathbb{N}\times \{1,2,..,d\}^\mathbb{N}\to \mathbb{R}$, and ask about the properties of the plan $\pi$ on $\{1,2,..,d\}^\mathbb{N}\times \{1,2,..,d\}^\mathbb{N}$ which minimize $\int \, c(x,y) d \pi(x,y)$ under the hypothesis that
the $y$-marginal of $\pi$ is a variable $\sigma$-invariant probability $\nu$, and the $x$-marginal of $\pi$ is a fixed probability $\mu$.

We note that a plain with this marginals properties is characterized by:
\begin{equation}
\left\{ \begin{array}{l}
\int f(x) \, d\pi(x,y) = \int f(x) \, d\mu(x) \ \ \text{for any} \, f\in C(X) \\
\int g(y) \,d\pi(x,y) = \int g(\sigma(y)) \, d\pi(x,y) \ \ \text{for any} \, g\in C(Y)
\end{array}
\right.
\label{marginal}
\end{equation}

We will show in section \ref{one} the following:

\begin{theorem}[\textbf{Kantorovich duality}]\label{dualidade}
Consider a compact metric space $X$ and $Y=\{1,2,..,d\}^\mathbb{N}$. Consider a fixed $\mu \in P(X)$, and a fixed continuous cost function $c:X\times Y\rightarrow \mathbb{R}^+$. Define $\Pi(\mu,\sigma)$ as the set of all Borel probabilities $\pi \in P(X\times Y)$ satisfying (\ref{marginal}). Define $\Phi_c$ as the set of all pair of continuous functions $(\varphi,\psi) \in C(X)\times C(Y)$ which satisfy:
\begin{equation}
\label{desigualdade phic}
\varphi(x)+\psi(y)- (\psi \circ \sigma)(y)\leq c(x,y), \ \ \ \ \forall \, (x,y) \in X\times Y
\end{equation}

Then,

I)
\begin{equation}
\label{resultado dualidade}
\inf_{\Pi(\mu,\sigma)} \int c \, d\pi = \sup_{(\varphi,\psi)\in\Phi_c} \int \varphi \, d\mu.
\end{equation}
Moreover, the infimum in the left hand side is attained.

II) If $c$ is a Lipschitz continuous function, then, there exist Lipschitz $\varphi$ and $\psi$ which are admissible realizers of the supremum.

\end{theorem}

Any pair $\varphi$ and $\psi$ satisfying (\ref{desigualdade phic}) is called admissible.
Any $\pi$ realizing the infimum in (\ref{resultado dualidade}) will be called an optimal plan, and its $y$-projection an optimal invariant probability solution for $c$ and $\mu$. Moreover, $\varphi$ and $\psi$ are called  an optimal dual pair if they realize the maximum of the right hand side expression. It is possible that does not exists an optimal dual pair (see remark bellow).

The following criteria is quite useful.

{\bf Slackness condition} \cite{Vi1} \cite{Vi2}: suppose for all $(x,y)$ in the support of  $\pi \in \Pi(\mu,\sigma)$ we have that
$$ \varphi(x)+\psi(y)- (\psi \circ \sigma)(y)= c(x,y),$$
for some admissible $\varphi$ and $\psi$ satisfying $\varphi+\psi- (\psi \circ \sigma)\leq c$, then
$\pi$ is an optimal plan and $\varphi$ and $\psi$ is  an optimal dual pair.

In recent years several results in the so called Ergodic Optimization Theory were obtained \cite{Jenkinson} \cite{CLT} \cite{Le}  \cite{Bousch1} \cite{Mo} \cite{Bousch-walters} \cite{GT1} \cite{BG}. We will show that the above kind of Ergodic Transport problem contains as a particular case this other theory. The subaction which possess properties of minimality described in \cite{CLT} and \cite{GL2} can be seen as a version of Kantorovich duality.

In the below remark we show that there are conceptual differences in the kind of analogous results we can get in the Classical and in the Ergodic Transport Theory.

\begin{remark} \label{dif}
We point out that if $\mu$ is a Dirac delta in a point $x_0$, then the cost $c(x_0,y)$ just depends on $y$.  In this way if we denote $A(y)=c(x_0,y)$ we get
that the above problem is the classical one in Ergodic Optimization, where one is interested in minimizing $\int  A \, d \nu$ among invariant probabilities $\nu$. There is no big difference in this theory if one consider maximization instead of minimization. The function $\psi$ above corresponds to the concept of subaction and the number $\varphi(x_0)$ is equal to $\min\{\int A(y) \, d\nu(y): \nu \, \text{is invariant}\}$ \cite{CLT}, \cite {Conze-Guivarch}, \cite{Bousch-walters}, \cite{Jenkinson}. It is known that for a $C^0$-generic continuous potential $A$  does not exist a continuous subaction \cite{BJ}. For the Classical Transport problem in compact spaces there exists continuous realizers for the dual problem when $c$ is continuous \cite{Vi2}. This shows that there are non trivial differences (at least for a $C^0$ cost function $c$) between the Classical and the Ergodic transport setting. It is known the existence of a calibrated Holder subaction for a Holder potential $A$. The item II on the above theorem is the correspondent result on the present setting. The expression $\varphi(x) +\psi(\sigma(y))-\psi(y) \leq c(x,y)$ can be for some people  more natural. This can be also obtained by replacing $\psi$ by $-\psi$.
\end{remark}

The next example shows that the Ergodic Transport problem can not be derived in an easy way from Ergodic Optimization properties.

We denote by $(a_{1}...a_{n})^\infty$ the periodic point $(a_{1}...a_{n}a_{1}...a_{n}a_{1}...)$ in $\{0,1\}^{\mathbb{N}}$.

\begin{example} Consider $X=\{x_{0},x_{1}\}$, $\mu = \frac{1}{2}(\delta_{x_{0}}+\delta_{x_{1}})$, $Y=\{0,1\}^{\mathbb{N}} $, and a cost function $c$ defined on $X\times Y$, satisfying the following proprieties: \newline
1) $c(x_{0},(01)^{\infty})=0$, $c(x_{0},(10)^{\infty})=1$, $c(x_{0},0^{\infty})=1/4$, $c(x_{0},y)>0$, if $y\neq (01)^{\infty}$. \newline
2) $c(x_{1},(01)^{\infty})=1$, $c(x_{1},(10)^{\infty})=0$, $c(x_{1},1^{\infty})=1/4$, $c(x_{1},y)>0$, if $y\neq (10)^{\infty}$. \newline
Assume $c$ is Lipschitz continuous.\newline
Note that, as an  example,  we can take $$c(x_{0},y)=d^{2}(y,(01)^{\infty}), \, c(x_{1},y)=d^{2}(y,(10)^{\infty}).$$

We observe that the measure $\nu=\frac{1}{2}(\delta_{(01)^{\infty}}+\delta_{(10)^{\infty}})$  is not a minimizing measure for either of  the potentials    $A_{0}(y):= c(x_0,y)$, or,  $A_{1}(y):= c(x_1,y)$. By the other hand, the unique optimal plan is given by $$\pi = \frac{1}{2}(\delta_{(x_{0},(01)^{\infty})} + \delta_{(x_1,(10)^{\infty})}),$$ which projects on $\mu$ and $\nu$.
\end{example}

We will also show in section \ref{uni} that generically on $c$ the optimal plan is unique.

In another kind of problem one can ask: given a continuous cost $c(x,y)$, $c:\{1,2,..,d\}^\mathbb{N}\times \{1,2,..,d\}^\mathbb{N}\to \mathbb{R}$, what are the properties of the probability $\pi$ on $\{1,2,..,d\}^\mathbb{N}\times \{1,2,..,d\}^\mathbb{N}$ which minimize $\int \, c(x,y) d \pi(x,y)$ under the hypothesis that
the $y$-marginal of $\pi$ is a variable invariant probability $\nu$ and the $x$-marginal of $\pi$ is a variable invariant probability $\mu$? Under what assumptions on $c$ we get that the optimal plan   $\pi$ is invariant for $\hat{\sigma}?$

We will present now formal definitions of the second class of problems.

Here we fix compact metric spaces $X$ and $Y$ and continuous transformations
\[T_{1}:X \times Y \to X, \ \ \ \ T_{2}: X \times Y \to Y,\]
such that, $T:X \times Y \to X\times Y$, given by $T=(T_{1},T_{2})$, defines a transformation of $X\times Y$ to itself.
Let $\Pi(T)$ the set of Borel probability measures $\pi$ in $X\times Y$, such that, for any $f:X\to \mathbb{R}$, $g:Y\to\mathbb{R}$:
$$\int f(x) \, d\pi(x,y) = \int f(T_{1}(x,y))\, d\pi(x,y), $$
and
$$ \int g(y) \, d\pi(x,y) = \int g(T_{2}(x,y))\, d\pi(x,y).$$

The set of such $\pi$ is called the set of admissible plans.

Note that any $T$-invariant measure in $X\times Y$ (which exists because $X\times Y$ is compact) satisfies this condition. Indeed, if $\nu$ is $T$-invariant, then:
\[ \int f(x) \, d\nu(x,y) = \int f(x,y) \, d\nu(x,y)=\]
\[=\int f(T(x,y)) \, d\nu(x,y)=\int f((T_{1}(x,y),T_{2}(x,y))) \, d\nu(x,y) = $$
$$\int f(T_{1}(x,y)) \, d\nu(x,y).\]

A similar reasoning can be applied to $g$.

Given a continuous function $c:X\times Y \to [0,+\infty)$, what can be said about
\[ \alpha(c) :=\inf \{\int c \, d\pi : \pi \in \Pi(T)\} ?\]

What are the properties of optimal plans? We are interested here in Kantorovich Duality type of results.

We will show the following:

\begin{theorem}  [\textbf{Kantorovich duality}] \label{duas}
$\alpha(c)$ is the supremum of the numbers $\alpha$ such that there exists continuous functions $\varphi:X\to\mathbb{R}$, $\psi:Y\to \mathbb{R}$ satisfying:
\[\alpha +\varphi(x)-\varphi(T_{1}(x,y)) + \psi(y)-\psi(T_{2}(x,y)) \leq c(x,y), \ \ \forall (x,y)\in X \times Y.\]
\end{theorem}

We can list different interesting cases where we can apply the above result:

\bigskip

1) If $T_{1}$ doesn't depends of $y\in Y$ and $T_{2}$ doesn't depends of $x\in X$, then we have the expression:
 \[\alpha +\varphi(x)-\varphi(T_{1}(x)) + \psi(y)-\psi(T_{2}(y)) \leq c(x,y) \ \ \forall (x,y)\in X \times Y.\]

In this case we are considering two variable invariant probabilities (one for $T_1$ and the other for $T_2$) as marginals of an admissible plan.


\bigskip

2) If $X$ and $Y$ are the Bernoulli space $\{1,2.,,d\}^\mathbb{N}$, $T_{1}=\sigma$ is the shift acting on the variable $x$, (doesn't depend of $y\in Y$) and $T_{2}=\tau_{x}(y)$ (where $\tau_j$, $j=1,2,...,d$, are the inverse branches of $\sigma$ acting on the variable $y$) we have that $T=\hat{\sigma}$ is the shift on $\{1,2.,,d\}^\mathbb{Z}$ and the above expression can be written as:
\[\alpha +\varphi(x)-\varphi(\sigma(x)) + \psi(y)-\psi(\tau_{x}(y)) \leq c(x,y) \ \ \forall (x,y)\in X \times Y.\]
In this case invariant probabilities $\pi$ for the shift $\hat{\sigma}:\{1,2.,,d\}^\mathbb{Z}\to \{1,2.,,d\}^\mathbb{Z}$ are admissible plans.
But not all admissible plan is $\hat{\sigma}$ invariant.

It is necessary to assume some special properties on $c$ in order to get an optimal plan which is $\hat{\sigma}$-invariant. This is the purpose of the next example.

\begin{example} Consider the points $x_{0}=(01)^{\infty}, x_{1}=(10)^{\infty}, \, y_{0}=(001)^{\infty}, y_{1}=(010)^{\infty}, y_{2}=(100)^{\infty}$. Let $c$ a Lipschitz continuous cost satisfying: \newline
1) $c(x_{0},y_{0}) = c(x_{1},y_{1})=c(x_{0},y_{2}) = c(x_{1},y_{2})=0$, \newline
2) $c>0$ on the other points.

\bigskip

It's easy to see that the unique optimal plan is given by:
$$\pi = \frac{1}{3}\delta_{(x_{0},y_{0})} + \frac{1}{3}\delta_{(x_{1},y_{1})} + \frac{1}{6}\delta_{(x_{0},y_{2})} + \frac{1}{6}\delta_{(x_{1},y_{2})}.$$
The support of this plan does not contain a $\hat{\sigma}$-invariant probability.
\end{example}

Which is the right assumption to get an optimal plan $\pi$  which is $\hat{\sigma}$-invariant? In
\cite{CLO}, \cite{GL4}, \cite{LOT}, \cite{LOS}, \cite{LO} some results in this direction are presented considering a cost which is dynamically defined.

\bigskip

Another class of examples:

\bigskip

3) If $X$ and $Y$ are the Bernoulli space $\{1,2.,,d\}^\mathbb{N}$,  and, for all $(x,y)\in X \times Y$, we have $T_{1}(x,y)=x$,  and $T_{2}(x,y)=\tau_{x_0}(y)$, $x=(x_0,x_1,...)\in \{1,2.,,d\}^\mathbb{N}$, (where $\tau_j$, $j=1,2,...,d$, are the inverse branches of $\sigma$), then, there is no $\varphi(x)$ in this case, and, the above expression can be written as:
\[\alpha +\psi(y)-\psi(\tau_{x}(y)) \leq c(x,y) \ \ \forall (x,y)\in X \times Y.\]
This is the holonomic setting of \cite{GL2}. A duality result is proved in section 2 in \cite{GL2}. Note that  the $y$-marginal of a holonomic probability is invariant for the shift $\sigma$ acting on the variable $y$ (see section 1 in \cite{GL2}). In the case $c$ is Holder it is shown the existence of Holder realizers $\psi$ in sections 3 and 4 \cite{GL2}.

\bigskip

We will show in sections \ref{zeta1} and \ref{zeta2} here that the optimal plans can be approximated by convex combination of optimal plans (of the classical transport problem) associated to measures supported in periodic orbits. In this way one can get an approximation scheme to the  optimal plan based on a finite set of conditions. Our approach here is inspired in  the point of view of taking the temperature going to zero for Gibbs states  at positive temperature in order to get results in  Ergodic Optimization \cite{LM1}.
The problem of fast approximation of maximizing probabilities by measures on periodic orbits plays an important role in Ergodic Optimization \cite{HY} \cite{BQ} \cite{CM}.

A paper which  consider Ergodic Transport problems under a continuous time setting is \cite{K1}.

We would like to thanks N. Gigli for very helpful comments and advice during the preparation of this manuscript.

\section{The case of one fixed probability and another variable invariant one} \label{one}

Here we will prove  Theorem \ref{dualidade}. We will adapt the main reasoning described in \cite{Vi1}.

Given a normed Banach space $E$ we denote by $E^{*}$ the dual space containing the bounded linear functionals from $E$ to $\mathbb{R}$.

We will need the following classical result \cite{Vi1} \cite{Vi2}.

\begin{theorem}[\textbf{Fenchel-Rockafellar duality}]
\label{Fenchel}
Suppose $E$ is  a vector normed space, $E^{\ast}$ its topological dual, $\Theta$ and $\Xi$ two convex functions defined on $E$ taking values in $\mathbb{R}\cup \{+\infty\}$. Denote $\Theta^{\ast}$ and  $\Xi^{\ast}$, respectively, the Legendre-Fenchel transform of  $\Theta$ and $\Xi$.
Suppose there exists  $x_0\in E$, such that $x_0 \in D(\Theta)\cap D(\Xi)$, and, that $\Theta$ is continuous on $x_0$.

Then,
\begin{equation}
\inf_{x \in E}[\Theta(x)+\Xi(x)]=\sup_{f\in E^{*}}[-\Theta^{*}(-f)-\Xi^{*}(f)] \label{rockafeller}
\end{equation}
Moreover, the supremum in ($\ref{rockafeller}$) is attained in at least one element in $E^*$.
\end{theorem}

\begin{proof}[\textbf{Proof of Theorem \ref{dualidade}}]

First we prove I).

We want to use  Fenchel-Rockafellar duality in the proof.

Define
$$E=C(X\times Y) \times M (Y)$$
where $C(X\times Y)$ is the set of all continuous functions in $X\times Y$ taking values in $\mathbb{R}$, with the usual sup norm $\parallel .\parallel_\infty$. Moreover, $M(Y)$ is the set of bounded linear operators in  $C(Y)$ taking values in $\mathbb{R}$ with the total variation norm.
Let $P_{\sigma}(Y)$ be the set of invariant probabilities  on $Y$.

Define $\Theta: E\longrightarrow \mathbb{R}\cup\{+\infty \}$ by
$$\Theta(u,\nu)=\left\{\begin{array}{ll}
0, \ \ \mbox{if} \ u(x,y)\geq -c(x,y), \ \forall(x,y)\in X\times Y, \mbox{and} \, \left\|\nu\right\| \leq 2\\
\\
+\infty, \ \ \mbox{in the other case},
\end{array}\right.
$$
Note that $\Theta$ is convex.

\bigskip

Define $\Xi:E\longrightarrow \mathbb{R}\cup\{+\infty \}$ by

$$
\Xi(u,\nu)=\left\{\begin{array}{ll}
\int_X\varphi d\mu, \ \ \mbox{if} \ u(x,y)=\varphi(x)+\psi(y)-(\psi \circ \sigma)(y), \\
\mbox{with} \ (\varphi,\psi)\in C(X)\times C(Y), \mbox{and} \, \nu \in P_{\sigma}(Y)\\
\\
+\infty, \ \ \mbox{in the other case}.
\end{array}\right.
$$

Note that  $\Xi$ is well defined. Indeed, if  $u= \varphi_{1}(x)+\psi_{1}(y)-(\psi_{1} \circ \sigma)(y) =\varphi_{2}(x)+\psi_{2}(y)-(\psi_{2} \circ \sigma)(y) $,
then, integrating under any invariant probability $\nu \in P_{\sigma}(Y)$, we have that
$\varphi_{1}(x)=\varphi_{2}(x)$.  Also note that $\Xi$ is convex. \newline
Observe that if $\nu \in P_{\sigma}(Y)$, then $(1,\nu)\in D(\Theta)\cap D(\Xi)$ and $\Theta$ is continuous in $(1,\nu)$.

Observe that
$$\inf_{(u,\nu)\in E}[\Theta(u,\nu)+\Xi(u,\nu)]$$
$$
=\inf\{\int_X\varphi d\mu :  \varphi(x)+[\psi -(\psi \circ \sigma)](y)\geq -c(x,y),
(\varphi,\psi)\in C(X)\times C(Y)\}$$
$$
=\inf\{-\int_X\varphi d\mu ;\,\, \,\varphi(x)+[\psi -(\psi \circ \sigma)](y)\leq c(x,y),
(\varphi,\psi)\in C(X)\times C(Y)\}$$
$$
=-\,\sup\{\int_X\varphi d\mu ;\,\, \,\varphi(x)+[\psi -(\psi \circ \sigma)](y)\leq c(x,y),
(\varphi,\psi)\in C(X)\times C(Y)\}$$
$$ =-\sup_{(\varphi,\psi)\in\Phi_c} \int \varphi \, d\mu.$$

Now we will compute the Legendre-Fenchel transform of $\Theta$ and $\Xi$, initially, for any $(\pi,g)\in E^{*}$: by the definition of $\Theta$
we get
\begin{eqnarray*}
\Theta^*((-\pi,-g)) &=&\sup_{(u,\nu)\in E}\left\{<(-\pi,-g),(u,\nu)>- \Theta(u,\nu) \right\} \\
&=& \sup_{(u,\nu)\in E}\left\{-\pi( u(x,y))- g(\nu): \ -u(x,y)\leq c(x,y), \left\| \nu \right\| \leq 2 \, \right \} \\
&=& \sup_{(u,\nu)\in E}\left\{\pi(u(x,y))- g(\nu) : \ u(x,y)\leq c(x,y) ,\left\| \nu \right\| \leq 2\right\}.\\
\end{eqnarray*}

Following \cite{Vi1} we note that if $\pi$ is not a positive functional, then, there exists a function $v \leq 0$, $v\in C(X\times Y)$, such that, $\pi(v)>0$, therefore, taking $u=\lambda v$ (remember that $c\geq 0$), and considering  $\lambda\rightarrow +\infty$, we get that $$\sup_{u\in E}\left\{\pi(u): \, u(x,y)\leq c(x,y) \right\}=+\infty.$$
When $\pi\in M^+(X\times Y)$ and $c\in C(X\times Y)$ we have that the supremum of $\pi(u)$ is, evidently, $\pi(c)$.

Therefore,
\begin{eqnarray}
\label{tetaestrela}
\Theta^*((-\pi,-g))=\left\{\begin{array}{ll}
\pi(c) + \displaystyle{\sup_{\left\|\nu\right\|\leq 2 } -g(\nu)},\ \mbox{if} \ \pi\in M^+(X\times Y) \\
\\
+\infty, \ \mbox{in the other case}.
\end{array}\right.
\end{eqnarray}

Analogously, by the definition of $\Xi$ we get that
$$\begin{array}{l}
\Xi^*(\pi,g)=\displaystyle{\sup_{(u,\nu)\in E}}\left\{<(\pi,g),(u,\nu)> - \Xi(u,\nu)\right\} \\ \\
=\displaystyle{\sup_{(u,\nu)\in E}}\left\{\begin{array}{lll} \,
\pi(u(x,y))-\int \varphi d\mu + g(\nu): \\
u(x,y)=\varphi(x)+\psi(y)- \psi (\sigma (y)) \ \mbox{where} \ (\varphi,\psi)\in C(x)\times C(Y)\\
\mbox{and}\, \nu \in P_{\sigma}(Y)
\end{array}\right\} \\ \\
=\displaystyle{\sup_{(\varphi,\psi)\in C(X)\times C(Y), \,\,\,\nu \in P_{\sigma}(Y)}\left\{ \pi(\varphi(x)+ \psi(y)- \psi (\sigma (y)))- \int \varphi d\mu + g(\nu)  \right\}.}
\end{array}$$

If $\pi(\varphi(x))\neq \int \varphi d\mu$ (we can suppose greater) for some $\varphi$, taking $\lambda.\varphi$ and $\lambda \to\infty$, the supremum will be equal to $+\infty$. Analogously if $\pi( \psi(y)- \psi (\sigma (y))) \neq 0$ (we can suppose greater) for some $\psi$, taking $\lambda.\psi$ and $\lambda \to \infty$, the supremum
will be $+\infty.$

In order to simplify the notation, define
\begin{eqnarray*}
  \Pi^*(\mu) = \left\{  \pi \in M(X\times Y): \, \begin{array}{ll} \pi(\varphi(x))=\int \varphi \,d\mu \, \, \text{and} \, \pi(\psi(y)- \psi(\sigma(y)))= 0 \\
     \, \forall (\varphi,\psi) \in C(X)\times C(Y)\end{array} \right\}.
\end{eqnarray*}

We  just show that
\begin{equation}
\label{xiestrela}
\Xi^*(\pi,g)=\left\{\begin{array}{lll}
\ \displaystyle{\sup_{\nu\in P_{\sigma}(Y)}g(\nu), \ \mbox{if} \  \pi\in\Pi^*(\mu),}\\
\\
+\infty, \ \mbox{in the other case}.
\end{array}\right.
\end{equation}
We know that the left hand side  (\ref{rockafeller}) is equal to
$-\sup_{(\varphi,\psi)\in\Phi_c} \int \varphi \, d\mu$, and also by (\ref{tetaestrela}) and (\ref{xiestrela}), we know that the right hand side of (\ref{rockafeller}) coincide with

$$\sup_{(\pi,g) \in E^{*}}\left\{
-\left(\begin{array}{ll}
\pi(c)+ \displaystyle{\sup_{\left\| \nu \right \| \leq 2} -g(\nu)} + \displaystyle{\sup_{\nu\in P_{\sigma}(Y)} g(\nu)},
 & \mbox{if} \ \pi\in M^+(X\times Y)\cap \Pi^*(\mu)\\
\\
+\infty  &  \mbox{in the other case}
\end{array}\right)
\right\}$$

$$= \sup_{(\pi,g) \in E^{*}}\left\{
\begin{array}{ll}
-\pi(c)+ \displaystyle{\inf_{\left\| \nu \right \| \leq 2} g(\nu)}  - \displaystyle{\sup_{\nu \in P_{\sigma}(Y)} g(\nu)}, &
 \mbox{if} \ \pi\in M^+(X\times Y)\cap \Pi^*(\mu)\\
\\
-\infty, & \mbox{in the other case}
\end{array}
\right\}$$

$$ = \sup \{-\pi(c), \ \pi\in M^+(X\times Y)\cap\Pi^*(\mu)\},$$
where the last equality is obtained taking $g=0$ because  $\left\|\nu \right\|=1$ for any $\nu \in P_{\sigma}(Y)$. \newline
Finally, we observe that if  $\pi \in  M^+(X\times Y)\cap\Pi^*(\mu,\nu)$ then: \newline
$\pi(1) = \mu(1) = 1$, \newline
$\pi(u)\geq 0$, if $u \geq 0$, \newline
$\pi$ is linear. \newline
From these properties we get that $\pi \in P(X\times Y)$. \newline
Moreover, by definition of $\Pi^*(\mu)$, the projection of $\pi$ in the first coordinate is $\mu$, and, the projection of $\pi$ is invariant in the second  coordinate. It follows that $M^+(X\times Y)\cap\Pi^*(\mu,\nu) = \Pi(\mu,\sigma)$.

\bigskip

Therefore,  from this  together with (\ref{rockafeller}) we get
\begin{equation*}
-\sup_{(\varphi,\psi)\in\Phi_c} \int \varphi \, d\mu = -\inf_{\pi \in \Pi(\mu,\sigma)} \int c \, d\pi
\end{equation*}
or,
\begin{equation*}
\sup_{(\varphi,\psi)\in\Phi_c} \int \varphi \, d\mu = \inf_{\pi \in \Pi(\mu,\sigma)} \int c \, d\pi.
\end{equation*}

 Note that theorem \ref{Fenchel} claims that
 $$\sup\limits_{f \in E^*}[-\Theta^*(-f)-\Xi^*(f)]=-\inf_{\pi \in \Pi(\mu,\sigma)} \int c \, d\pi$$
 is attained, for at least one element, and this shows the existence of an optimal plan. This shows I).

After we get the probability $\nu$ we can consider the classical transport problem for $\mu$, $\nu$ and $c$, and, finally,  we can  get some well known properties described in the classical literature (as, slackness condition, $c$-cyclical monotonicity, etc...).

 \bigskip

 Now, we will prove II).
 This will follow from the following claim.

{\bf Claim:}
Let $X$ be a compact metric space,  $Y=\{1,...,d\}^{\mathbb{N}}$, $c:X\times Y \to \mathbb{R}$ be a Lipschitz continuous function and $\mu$ be a probability measure acting in $X$. Let $\pi \in \Pi(\mu,\sigma)$ minimizing the integral of $c$. Then, there exist Lipschitz continuous functions $\varphi(x), \psi(y)$ such that:
\newline
i) $\varphi(x) + \psi(\sigma(y)) -\psi(y) \leq c(x,y)$ \newline
ii) $\int \varphi(x) \, d\mu = \int c \, d\pi.$

\bigskip

Let $\beta$ the Lipschitz constant of $c$. \newline
First note that  given continuous functions $\varphi$ and $\psi$  satisfying
\[\varphi(x) + \psi(\sigma(y)) -\psi(y) \leq c(x,y),\]
then, there exists $\overline{\varphi}$ and $\overline{\psi}$,  Lipschitz functions with Lipschitz constant $\beta$ satisfying:
\[\overline{\varphi}(x) + \overline{\psi}(\sigma(y)) -\overline{\psi}(y) \leq c(x,y),\]
and,
\[\overline{\varphi} \geq \varphi.\]
We can choose $\overline{\psi}$ satisfying $0\leq \overline{\psi}\leq \beta$.

\bigskip

Indeed, for any $\sigma^{n}(w) = y$ and $x_{0},...,x_{n-1} \in X$:
\[\psi(y) - \psi(w) \leq \sum_{i=0}^{n-1}c(x_{i},\sigma^{i}(w)) - \varphi(x_{i}).\]
This shows that
\[\overline{\psi}(y) := \inf\{ \sum_{i=0}^{n-1}c(x_{i},\sigma^{i}(w)) - \varphi(x_{i}): \, n\geq 0, \, \sigma^{n}(w)=y, \, x_{i}\in X\}\]
is well defined.
\newline
We remark that $\overline{\psi}$ is a Lipschitz function with the same constant $\beta$. Note also that
\[\varphi(x) + \overline{\psi}(\sigma(y)) -\overline{\psi}(y) \leq c(x,y).\]
Now for each $x$ fixed, define $\overline{\varphi}(x)$ as the greatest number such that for any $y$:
\[\overline{\varphi}(x) + \overline{\psi}(\sigma(y)) -\overline{\psi}(y) \leq c(x,y).\]
We note that $\overline{\varphi}\geq \varphi$ and $\displaystyle{\overline{\varphi}(x)=\inf_{y}\{c(x,y) +\overline{\psi}(y) - \overline{\psi}(\sigma(y))\}}$.
We also note that $\overline{\varphi}$ is a Lipschitz function with same constant $\beta$. We remark that we can add a constant to $\overline{\psi}$, and, so we can suppose without lost of generality that $0\leq \overline{\psi}\leq \beta$.
\bigskip

Now, we prove the main claim.

By (\ref{resultado dualidade}), there exists sequences of continuous functions $\varphi_{n}$ and $\psi_{n}$, $n \in \mathbb{N}$,  such that
\[\varphi_{n}(x) + \psi_{n}(\sigma(y)) -\psi_{n}(y) \leq c(x,y)\]
and
\[\int \varphi_{n} \, d\mu \to \int c \, d\pi.\]

From the above reasoning we can get $\overline{\varphi}_{n}, \overline{\psi}_{n}$ Lipschitz continuous functions such that
\[\overline{\varphi}_{n}(x) + \overline{\psi}_{n}(\sigma(y)) -\overline{\psi}_{n}(y) \leq c(x,y),\]
and,
\[\int \overline{\varphi}_{n} \, d\mu \to \int c \, d\pi.\]
For a fixed $\epsilon>0$, we get
\[\int c \, d\pi -\epsilon < \int \overline{\varphi}_{n} \, d\mu \leq \int c \, d\pi,\]
for $n$ large enough.
Particularly, for a fixed $n$ sufficiently large, there exist $x_{n}, x'_{n}\in X$ such that
\[ \int c \, d\pi -\epsilon \leq \overline{\varphi}_{n}(x_{n}) \ \ \ \  \text{and} \ \ \ \ \overline{\varphi}_{n}(x'_{n})  \leq \int c \, d\pi.\]
Using the fact that $\overline{\varphi}_{n}$ is Lipschitz continuous, with constant $\beta$, and, denoting by $D$ the diameter of $X$, we conclude that for large $n$:
\[\int c \, d\pi -\epsilon - D\,\beta \leq \overline{\varphi}_{n}  \leq \int c \, d\pi +D\,\beta.\]
So we can apply the Arzela-Ascoli theorem and, finally, we get continuous functions $\varphi$, $\psi$ satisfying:\newline
i) $\varphi(x) + \psi(\sigma(y)) -\psi(y) \leq c(x,y)$ \newline
ii) $\int \varphi(x) \, d\mu = \int c \, d\pi.$ \newline

We know from the first reasoning that we can assume $\varphi$ and $\psi$ are Lipschitz continuous functions. This shows II).

\end{proof}

\section{Generic properties: a unique optimal plan} \label{uni}

\begin{lemma}
Let $K$ be a compact set in $\mathbb{R}^{2}$, and for each $r>0$, define $K_{r}$ as the set of points $(x,y)\in K$ such that $x+ry$ is maximal. Then de diameter of $K_{r}$ converges to zero when $r\to 0$.
\end{lemma}
\begin{proof}
See  \cite{Bousch-walters}  page 306 for the proof.
\end{proof}

\begin{corollary}\label{bousch}
With the hypothesis of  the above lemma, for each $\epsilon>0$, there exists a $r_{0}>0$, such that, for $r_{0}>r>0$ and $(x_{1},y_{1}), \,(x_{2},y_{2}) \in K_{r}$, we have:
\[|y_{1}-y_{2}| <\epsilon.\]
\end{corollary}

The bellow theorem follows from the same arguments used in  proposition 9 in \cite{Bousch-walters}.

\begin{theorem}\label{teogeneric} Let $X$ be a compact metric space, $Y=\{1,...,d\}^{\mathbb{N}}$ and $\mu$ be a probability measure in $X$. Let $C(X,Y)$ be the set of continuous functions from $X\times Y$ to $\mathbb{R_{+}}$ with the uniform norm. The set of functions $c \in C(X,Y)$ with a unique Optimal Plan in $\Pi(\mu,\sigma)$ is generic in $C(X,Y)$. The same is true for the Banach space $H(X,Y)$ of the Lipschitz functions with the usual norm.

\end{theorem}
\begin{proof}
On this proof, we are going to consider $\pi$ an optimal plan if $\int c d\pi$ is maximal (just consider the change of $c$ by $-c$). \newline
We start studying the space $C(X,Y)$. Given an countable family $(e_{i})_{i\in\mathbb{N}}$ dense in $C(X,Y)$, the set of functions in $C(X,Y)$ with two or more optimal plans coincides with:
\[\bigcup_{m,n \in \mathbb{N}} X_{m,n},\]
where
\[X_{m,n}:=\{c \in C(X,Y) : \, \exists \, \pi,\chi \, \text{optimal plans}, \, \int e_{n} \, d(\pi-\chi) \geq \frac{1}{m}\}.\]
Then it's sufficiently to prove that $X_{m,n}$ is a closed set with empty interior.

\bigskip

\textbf{Claim 1:}  $X_{m,n}$ is a closed set.
\newline
Indeed, we note that $C(X,Y)$ is a normed space. Consider $c_{s}$ in $X_{m,n}$ converging to $c$ (when $s\to\infty$). Let $(\pi_{s},\chi_{s})$ be the optimal ones associated with $c_{s}$ in $X_{m,n}$. We can suppose, by taking a subsequence, that $\pi_{s}\to \pi$, $\chi_{s} \to \chi$, where $\pi, \chi$ are probability measures in $X\times Y$. So
 $$\int e_{n} \, d(\pi-\chi) = \lim_{s\to\infty}\int e_{n} \, d(\pi_{s}-\chi_{s})\geq \frac{1}{m}.$$
 Clearly $\pi,\chi \in \Pi(\mu,\sigma)$. Also, by the above relation, they are different measures.   We want to show that the limit function $c$ is in $X_{m,n}$. We only need to prove that $\pi$ and $\chi$ are optimal plans to $c$. Suppose by contradiction there exists $\zeta \in \Pi(\mu,\sigma)$ such that $\int c \, d\zeta > \int c \, d\pi + \epsilon$. So for $s$ large we have:
\[\int c_{s} \, d\zeta > \int c \, d\zeta - \epsilon/3 > \int c \, d\pi + 2\epsilon/3 > \int c \, d\pi_{s} +\epsilon/3 >  \int c_{s} \, d\pi_{s}.\]
This is impossible because $\pi_{s}$ is an optimal plan for $c_{s}$. Therefore, $\chi$ is an optimal plan for $c$.

\bigskip

\textbf{Claim 2:} $X_{m,n}$ has empty interior.
\newline
Indeed, for a fixed $c\in X_{m,n}$ we can show that $c+re_{n}\notin X_{m,n}$ when $r>0$ is sufficiently small. Consider
$$K= \left\{ \left( \int c \, d\pi, \int e_{n}\, d\pi \right): \pi \in \Pi(\mu, \sigma) \right\}.$$
$K$  is compact and contained in $\mathbb{R}^{2}$.
Then, by the Corollary \ref{bousch}, when $\epsilon= 1/2m$, there exist a $r_{0}$, such that for $r_{0}>r>0$ we get: if $\int (c+r\epsilon_{n})\, d\pi$ and $\int (c+r\epsilon_{n})\, d\chi$ are maximal (this means $\pi$, $\chi$ optimal plans to $f+r\epsilon_{n}$), then $\left|\int (\epsilon_{n})\, d(\pi-\chi)\right| < \varepsilon =1/2m$. This show that $c+r\epsilon_{n} \notin X_{m,n}$.

\bigskip

In the space $H(X,Y)$ we can get similar results. This can be obtained with  the same arguments used before together with the following remarks: \newline
 a) A dense enumerable family $\{e_{n}\}$ in $H(X,Y)$ will be a dense sequence in $C(X,Y)$. In this way two elements $\pi, \chi \in \Pi(\mu,\sigma)$ will be different, if and only if, $\int e_{n} \, d(\pi-\chi) \neq 0$, for some $e_{n}$.  \newline
b) Moreover, the set \[X_{m,n}:=\{c \in H(X,Y) : \, \exists \, \pi,\chi \, \text{optimal plans}, \, \int e_{n} \, d(\pi-\chi) \geq \frac{1}{m}\}\] will be a closed set by the same arguments used above. We can also show, in  a similar way, that it has empty interior, but we note that we need $c+re_{n} \in H(X,Y)$, and $c+re_{n} \to c$ in Lipschitz norm. This is true because we can consider ${e_{n}} \in H(X,Y)$.

\end{proof}

\section{Zeta-measures and Transport} \label{zeta1}

When $\mu, \nu$ have support in a finite number of points, then,  the optimal plan $\pi(\mu,\nu)$ for a cost $c$ in the Classical Transport Theory can be explicitly obtained by Linear Algebra arguments \cite{Vi1}. Indeed, suppose
$$\mu = a_{1}\delta_{x_{1}}+...+a_{n}\delta_{x_{n}},$$ and,
$$\nu= b_{1}\delta_{y_{1}}+...+b_{m}\delta_{y_{m}}.$$

 Then any transport plan $\pi$ have support contained in $\{x_{1},...x_{n}\}\times \{y_{1},...,y_{m}\}$. Denoting by $\pi_{ij}$ the mass of $\pi$ in $(x_{i},y_{j})$ we have that the variables $\pi_{ij}$ need satisfies the linear equations:
\newline
vertical equations:
\begin{align*}
a_{1} &= \pi_{11}+...+\pi_{m1}\\
...\\
a_{n} &= \pi_{1n}+...+\pi_{mn}
\end{align*}
horizontal equations:
\begin{align*}
b_{1} &= \pi_{11}+...+\pi_{1n}\\
...\\
b_{m} &= \pi_{m1}+...+\pi_{mn}
\end{align*}

The set of solutions of this equations defines a convex in $\mathbb{R}^{n.m}$. The conditions $\pi_{ij}\geq 0$ will restrict the solutions to a bounded convex set with finite vertexes. So given a cost function $c:X\times Y\to[0,\infty)$, denoting their restricted values by $c_{ij}:=c(x_{i},y_{j})$, we have that the optimal plans to $(\mu,\nu)$ are the point of the above convex set that minimize the linear functional:
\[ \sum_{i,j}c_{ij}\pi_{ij}.\]
By convexity arguments there is an optimal point in the vertexes of the underlying convex set. The conclusion is that by Linear Algebra arguments we can find a finite number of points such that at least one of these will be optimal for the integral of $c$ with the given marginals $\mu$ and $\nu$.

 Note that these finite vertexes points are determined before we consider the given cost function.

It is natural in the Ergodic Theory setting to try to approximate a general invariant probability by the ones which posses the simplest behavior: the periodic probabilities. These are the ones that we can make computations more easily.

We note that to minimize the integral of the cost function $c$ is the same that to maximize the integral of the function $-c$. The plan that realizes this optimal integral will the same if we add a constant to $-c$. Bellow we consider the problem of finding a transport plan maximizing the integral of a cost $c$ strictly greater than zero. A transport plan from $\mu$ to $\nu$ maximizing the integral of $c$ will be called a \textbf{maximizing plan}.
Below we consider a compact metric space $X$ and $Y=\{1,...,d\}^{\mathbb{N}}$.

\begin{definition}
For fixed $\mu\in P(X)$ and a continuous function $c:X\times Y \to \mathbb{R}$ we define a probability measure in $X\times Y$ by the linear functional $\zeta_{\beta,n}: C(X\times Y) \to \mathbb{R}$, such  that, for each $w \in C(X\times Y)$ associate the number:
\[ \frac{\sum_{\nu \in Fix_{n}}e^{\beta.n.\int c(x,y) d\pi(\mu,\nu)}\int w \, d\pi(\mu,\nu)}{\sum_{\nu \in Fix_{n}}e^{\beta.n.\int c(x,y) d\pi(\mu,\nu)}},\]
where $Fix_{n}$ denotes the set of invariant measures in $Y$ supported in a periodic orbit of length $n$, and, $\pi(\mu,\nu)$ denotes a maximizing plan from $\mu$ to $\nu$ with cost function $c$ (we don't impose other conditions on the chosen the plan).
\end{definition}

In the case $\mu$ is supported in a unique point $x_{0}$, we can define the function $A(y) = c(x_{0},y)$, and this measure can be written as:
\[  \frac{\sum_{y \in Fix_{n}}e^{\beta.A^{n}(y)}\frac{ \overline{w}^{n}(y)}{n}}{\sum_{y \in Fix_{n}}e^{\beta.A^{n}(y)}}\]
where $\overline{w}(y) = w(x_{0},y)$ and $A^{n}(z) = A(z) + ...+A(\sigma^{n-1}(z))$. This kind of measure (also called zeta measure) is considered in Thermodynamical Formalism \cite{PP} and they can be used to approximate Gibbs states, and, also the measure that maximizes the integral of $A$ among the invariant measures (see \cite{LM1}). Therefore, in some sense, the above defined family of probabilities extend a well known concept used  in Ergodic Optimization.

\bigskip

In the case of that $\mu$ have finite support these zeta-measures can be determined by Linear Algebra arguments like we remarked above.

\bigskip

Remember that $\Pi(\mu,\sigma)$ is the set of probabilities measures that coincides with $\mu$ in the first coordinate and is invariant in the second coordinate. The next theorem follows the ideas used in thermodynamical limit when $\beta, n \to \infty$ \cite{LM1}.

\begin{theorem}
When $\beta, n$ goes to infinite, any limit measure $\pi_{\infty}$ of convergent subsequence of $\zeta_{\beta,n}$, in the weak* topology, belongs to $\Pi(\mu,\sigma)$. Moreover, if $c>0$, then, $\pi_{\infty}$ maximizes the integral of $c$ among the measures in $\Pi(\mu,\sigma)$.
\end{theorem}

\begin{proof}
We begin by proving that for $\beta,n$ fixed, the corresponding zeta-measure is in $\Pi(\mu,\sigma)$. Let $w$ a function depending only on x. Then
\[ \frac{\sum_{\nu \in Fix_{n}}e^{\beta.n.\int c(x,y) d\pi(\mu,\nu)}\int w \, d\pi(\mu,\nu)}{\sum_{\nu \in Fix_{n}}e^{\beta.n.\int c(x,y) d\pi(\mu,\nu)}}\]
\[= \frac{\sum_{\nu \in Fix_{n}}e^{\beta.n.\int c(x,y) d\pi(\mu,\nu)}\int w \, d\mu}{\sum_{\nu \in Fix_{n}}e^{\beta.n.\int c(x,y) d\pi(\mu,\nu)}}\]
\[=\int w \, d\mu \frac{\sum_{\nu \in Fix_{n}}e^{\beta.n.\int c(x,y) d\pi(\mu,\nu)}}{\sum_{\nu \in Fix_{n}}e^{\beta.n.\int c(x,y) d\pi(\mu,\nu)}}\]
\[=\int w \, d\mu.\]
Now, consider a  fixed  function $w$ depending only of y. Then,  we have:
\[ \zeta_{\beta,n}(w\circ\sigma)=\frac{\sum_{\nu \in Fix_{n}}e^{\beta.n.\int c(x,y) d\pi(\mu,\nu)}\int w\circ\sigma \, d\pi(\mu,\nu)}{\sum_{\nu \in Fix_{n}}e^{\beta.n.\int c(x,y) d\pi(\mu,\nu)}}\]
\[= \frac{\sum_{\nu \in Fix_{n}}e^{\beta.n.\int c(x,y) d\pi(\mu,\nu)}\int w\circ\sigma \, d\nu}{\sum_{\nu \in Fix_{n}}e^{\beta.n.\int c(x,y) d\pi(\mu,\nu)}}\]
\[=\frac{\sum_{\nu \in Fix_{n}}e^{\beta.n.\int c(x,y) d\pi(\mu,\nu)}\int w \, d\nu}{\sum_{\nu \in Fix_{n}}e^{\beta.n.\int c(x,y) d\pi(\mu,\nu)}}\]
\[=\frac{\sum_{\nu \in Fix_{n}}e^{\beta.n.\int c(x,y) d\pi(\mu,\nu)}\int w \, d\pi(\mu,\nu)}{\sum_{\nu \in Fix_{n}}e^{\beta.n.\int c(x,y) d\pi(\mu,\nu)}} = \zeta_{\beta,n}(w).\]
This show that $\zeta_{\beta,n} \in \Pi(\mu,\sigma)$. So, when $\beta, n$ goes to infinite, any limit measure $\pi_{\infty}$ of convergent subsequence of $\zeta_{\beta,n}$ is on $\Pi(\mu,\sigma)$.

\bigskip

Suppose $\zeta_{\beta_j,n_j}\to \pi_{\infty}$, when $j\to\infty$.

Let $\pi^* \in \Pi(\mu,\sigma)$ maximizing the integral of $c$. Let $\nu^*$ the projection of $\pi^*$ on the second coordinate $y$. Then, $\nu^*$ is a invariant measure. Let $\nu_{n_j}\in Fix_{n_j}$ be a subsequence converging to $\nu^*$ in the weak* topology. If $\pi_{n_j}$ is a maximizing plan from $\mu$ to $\nu_{n_j}$, then, there exist a subsequence $\pi_{n_i}$ converging to a maximizing plan $\pi$ from $\mu$ to $\nu^*$ (\cite{Vi2} page 77). It is easy to see that
\[ \int c \, d\pi = \int c d \pi^*, \]
and, therefore,  $\pi$ is maximal. In other words,  it maximizes the integral of $c$ among the measures in $\Pi(\mu,\sigma)$. We denote this integral by $I(c) $. We want to show that $\pi_{\infty}(c) \geq I(c)$, where the subsequence $\zeta_{\beta_i,n_i}$ converges to $\pi_{\infty}$ in the weak* topology.
From the above arguments we know that:
 \bigskip
\begin{center}
\textsl{Given $\varepsilon >0$, for sufficiently large $i$ there exist $\nu \in Fix_{n_i}$ such that:} $$\int c \, d\pi(\mu,\nu) > I(c) - \varepsilon.$$
\end{center}

Take $\varepsilon >0$, such that $(I(c)-\varepsilon)>0$, and  define:
\[A_{n_i}(\varepsilon) =\{ \nu \in Fix_{n_i} : \int c \, d\pi(\mu,\nu) \leq I(c) - \varepsilon\}\]
\[B_{n_i}(\varepsilon) = \{ \nu \in Fix_{n_i} : \int c \, d\pi(\mu,\nu) > I(c) - \varepsilon\}.\]

Then, we have:
\[\sum_{\nu \in A_{n_i}(\varepsilon)}e^{\beta_i.n_i.\int c(x,y) d\pi(\mu,\nu)} \leq \sum_{\nu \in A_{n_i}(\varepsilon)}e^{\beta_i.n_i.(I(c) -\varepsilon)} \]
\[\leq e^{n_i\log(d)+\beta_i.n_i.(I(c) -\varepsilon)},\]
and
\[ \sum_{\nu \in A_{n_i}(\varepsilon)}e^{ \beta_{i}.n_{i}.\int c(x,y) d\pi(\mu,\nu)}\int c \, d\pi(\mu,\nu) \leq  e^{n_i\log(d)+\beta_i.n_i.(I(c) -\varepsilon)}(I(c)-\varepsilon).\]
By other hand, if $n_i$ is sufficiently large, $B_{n_i}(\varepsilon/2)$ is non empty. It follows that
\[ \sum_{\nu \in B_{n_i}(\varepsilon)}e^{\beta_i.n_i.\int c(x,y) d\pi(\mu,\nu)} \geq \sum_{\nu \in B_{n_i}(\varepsilon/2)}e^{\beta_i.n_i.\int c(x,y) d\pi(\mu,\nu)}\]
\[\geq e^{\beta_i.n_i.(I(c) -\varepsilon/2)},\]
and,
\[\sum_{\nu \in B_{n_i}(\varepsilon)}e^{\beta_i.n_i.\int c(x,y) d\pi(\mu,\nu)}\int c \, d\pi(\mu,\nu) \geq  e^{\beta_i.n_i.(I(c) -\varepsilon/2)}(I(c)-\varepsilon/2).\]
Then, we get
\[ 0\leq \lim_{i \to\infty} \frac{\sum_{\nu \in A_{n_i}(\varepsilon)}e^{\beta_i.n_i.\int c(x,y) d\pi(\mu,\nu)}}{\sum_{\nu \in B_{n_i}(\varepsilon)}e^{\beta_i.n_i.\int c(x,y) d\pi(\mu,\nu)}} \]
\[\leq \lim_{i \to\infty} \frac{e^{n_i\log(d)+\beta_i.n_i.(I(c) -\varepsilon)}}{e^{\beta_i.n_i.(I(c) -\varepsilon/2)}} = \lim_{i \to\infty}e^{n_i\log(d) - \beta_i.n_i.\varepsilon/2} =0.\]
Moreover,
$$0 \leq \lim_{i \to\infty} \frac{\sum_{\nu \in A_{n_i}(\varepsilon)}e^{\beta_i.n_i.\int c(x,y) d\pi(\mu,\nu)}\int c \, d\pi(\mu,\nu)}{\sum_{\nu \in B_{n_i}(\varepsilon)}e^{\beta_i.n_i.\int c(x,y) d\pi(\mu,\nu)}\int c \, d\pi(\mu,\nu)} \leq $$
$$\lim_{i \to\infty} \frac{e^{n_i\log(d)+\beta_i.n_i.(I(c) -\varepsilon)}(I(c)-\varepsilon)}{e^{\beta_i.n_i.(I(c) -\varepsilon/2)}(I(c)-\varepsilon/2)}=$$
$$ \lim_{i \to\infty}e^{n_i\log(d) - \beta_i.n_i.\varepsilon/2}\frac{I(c) -\varepsilon}{I(c)-\varepsilon/2} = 0.$$
Finally,
\[\liminf_{i\to\infty} \frac{\sum_{\nu \in Fix_{n_i}}e^{\beta_i.n_i.\int c(x,y) d\pi(\mu,\nu)}\int c \, d\pi(\mu,\nu)}{\sum_{\nu \in Fix_{n_i}}e^{\beta_i.n_i.\int c(x,y) d\pi(\mu,\nu)}}\]
\[=\liminf_{i\to\infty} \frac{ \sum_{\nu \in B_{n_i}(\varepsilon)}e^{\beta_i.n_i.\int c(x,y) d\pi(\mu,\nu)}\int c \, d\pi(\mu,\nu)}{\sum_{\nu_i \in B_{n}(\varepsilon)}e^{\beta_i.n_i.\int c(x,y) d\pi(\mu,\nu)}} \]
\[\geq \liminf_{i\to\infty} \frac{ \sum_{\nu \in B_{n_i}(\varepsilon)}e^{\beta_i.n_i.\int c(x,y) d\pi(\mu,\nu)}(I(c)-\varepsilon)}{\sum_{\nu \in B_{n_i}(\varepsilon)}e^{\beta_i.n_i.\int c(x,y) d\pi(\mu,\nu)}} \]
\[\geq I(c)-\varepsilon.\]

Taking $\varepsilon \to 0$, we get:
\[\liminf_{i\to\infty} \frac{\sum_{\nu \in Fix_{n_i}}e^{\beta_i.n_i.\int c(x,y) d\pi(\mu,\nu)}\int c \, d\pi(\mu,\nu)}{\sum_{\nu \in Fix_{n_i}}e^{\beta_i.n_i.\int c(x,y) d\pi(\mu,\nu)}} \geq I(c).\]
Using the fact that $\zeta_{\beta_i,n_i}\to \pi_{\infty}$ we conclude that $\pi_{\infty}(c)\geq I(c)$.
\end{proof}

\section{Two variable invariant probabilities and other cases}

We start this section by proving Theorem \ref{duas}. The proof follows basically the same kind of ideas that were used in Theorem \ref{dualidade}.

\begin{proof}

Define
$$E=C(X\times Y) \times M (X\times Y)$$
where $M(X\times Y)$ is the set of bounded linear functionals from $C(X\times Y)$ to $\mathbb{R}$ with the norm given by the total variation.

Let $\Theta: E\longrightarrow \mathbb{R}\cup\{+\infty \}$ given by
$$\Theta(u,\nu)=\left\{\begin{array}{ll}
0, \ \ \mbox{if} \ u(x,y)\geq -c(x,y), \ \forall(x,y)\in X\times Y, \mbox{and} \, \left\|\nu\right\| \leq 2\\
\\
+\infty, \ \ \mbox{in the other case},
\end{array}\right.
$$
Note that $\Theta$ is a convex function.

\bigskip

Define $\Xi:E\longrightarrow \mathbb{R}\cup\{+\infty \}$ by
$$
\Xi(u,\nu)=\left\{\begin{array}{ll}
\alpha, \ \mbox{if} \ u(x,y)=\alpha +\varphi(x)-\varphi(T_{1}(x,y)) + \psi(y)-\psi(T_{2}(x,y)), \\
\mbox{with} \ (\varphi,\psi)\in C(X)\times C(Y), \mbox{and} \, \nu \in \Pi(T)\\
\\
+\infty, \ \ \mbox{in the other case}.
\end{array}\right.
$$

We remark that $\Xi$ is a well defined convex function.

If $\nu \in \Pi(T)$, then $(1,\nu)\in D(\Theta)\cap D(\Xi)$ and $\Theta$ is continuous in $(1,\nu)$.

 \bigskip

Note that:

$$\inf_{(u,\nu)\in E}[\Theta(u,\nu)+\Xi(u,\nu)]$$
$$
=\inf \left\{ \alpha : \, \begin{array}{l} \exists (\varphi,\psi)\in C(X)\times C(Y), \\
\alpha + \varphi(x)- \varphi(T_{1}(x,y))+\psi(y) -\psi(T_{2}(x,y))\geq -c(x,y) \end{array} \right\}$$
$$
=\inf \left\{-\alpha : \, \begin{array}{l} \exists (\varphi,\psi)\in C(X)\times C(Y), \\
\alpha + \varphi(x)- \varphi(T_{1}(x,y))+\psi(y) -\psi(T_{2}(x,y))\leq c(x,y) \end{array}
\right\}$$
$$
=-\,\sup \left\{\alpha : \, \begin{array}{l} \exists (\varphi,\psi)\in C(X)\times C(Y), \\
 \alpha + \varphi(x)- \varphi(T_{1}(x,y))+\psi(y) -\psi(T_{2}(x,y))\leq c(x,y) \end{array}
\right\}.$$

\bigskip

So the left size of (\ref{rockafeller}) is:
\begin{equation}
 -\,\sup\left\{\alpha :\,\, \begin{array}{l} \exists (\varphi,\psi)\in C(X)\times C(Y), \\
 \alpha + \varphi(x)- \varphi(T_{1}(x,y))+\psi(y) -\psi(T_{2}(x,y))\leq c(x,y) \end{array}
\right\}. \label{leftsize}
\end{equation}

Now we will compute the Legendre-Fenchel transform of $\Theta$ and $\Xi$. For each $(\pi,g)\in E^{*}$:
\begin{eqnarray*}
\Theta^*((-\pi,-g)) &=&\sup_{(u,\nu)\in E}\left\{<(-\pi,-g),(u,\nu)>- \Theta(u,\nu) \right\} \\
&=& \sup_{(u,\nu)\in E}\left\{-\pi( u(x,y))- g(\nu): \ -u(x,y)\leq c(x,y), \left\| \nu \right\| \leq 2 \, \right \} \\
&=& \sup_{(u,\nu)\in E}\left\{\pi(u(x,y))- g(\nu) : \ u(x,y)\leq c(x,y) ,\left\| \nu \right\| \leq 2\right\}.\\
\end{eqnarray*}

Following \cite{Vi1} note that if $\pi\notin M^{+}(X\times Y)$,  then there exists a function $v \leq 0$ in $C(X\times Y)$, such that, $\pi(v)>0$, so taking $u=\lambda v$ (remember that $c\geq 0$), when $\lambda\rightarrow +\infty$, we have that $\sup_{u\in E}\left\{\pi(u): \, u(x,y)\leq c(x,y) \right\}=+\infty$.

Moreover, if $\pi\in M^+(X\times Y)$, using the fact that $c\in C(X\times Y)$, then we have that the maximum of $\pi(u)$ is given by $\pi(c)$.

Therefore,
\begin{eqnarray}
\label{tetaestrela2}
\Theta^*((-\pi,-g))=\left\{\begin{array}{ll}
\pi(c) + \displaystyle{\sup_{\left\|\nu\right\|\leq 2 } -g(\nu)},\ \mbox{if} \ \pi\in M^+(X\times Y) \\
\\
+\infty, \ \mbox{in the other case}.
\end{array}\right.
\end{eqnarray}

Now we analyze $\Xi^{*}$:
$$
\Xi^*(\pi,g)=\sup_{(u,\nu)\in E}\left\{<(\pi,g),(u,\nu)> - \Xi(u,\nu)\right\}$$
$$=\sup_{(u,\nu)\in E}\left\{\begin{array}{lll} \,
\pi(u(x,y))-\alpha + g(\nu): \\
u(x,y)=\alpha + \varphi(x)- \varphi(T_{1}(x,y))+\psi(y) -\psi(T_{2}(x,y)) ,\\
\mbox{with} \ (\varphi,\psi)\in C(x)\times C(Y)\, \mbox{and}\, \nu \in \Pi(T)
\end{array}\right\} $$
$$= \sup_{\alpha,(\varphi,\psi)\in C(X)\times C(Y), \,\nu \in \Pi(T)} \pi(\alpha)-\alpha+\pi(\varphi-\varphi\circ T_{1})+ \pi(\psi- \psi\circ T_{2})+ g(\nu)  .$$

 If $\pi(\varphi(x)-\varphi(T_{1}(x,y)))\neq 0$ (we can suppose greater than zero) for some $\varphi$, by taking $\lambda.\varphi$ and $\lambda \to\infty$, the supremum will be equal to $+\infty$. Analogously, if $\pi( \psi(y)- \psi (T_{2}(x,y))) \neq 0$, the supremum will be $+\infty.$  If $\pi(1)\neq 1$ (we can suppose greater than one), then taking $\alpha\to\infty$, the supremum will be $+\infty$.

Define
\begin{eqnarray*}
  \Pi^*(T) = \left\{  \pi \in M(X\times Y): \, \begin{array}{ll} \pi(1)=1,\, \pi(\varphi-\varphi\circ T_{1})=\pi(\psi- \psi\circ T_{2})= 0  ,\\
\, \forall (\varphi,\psi) \in C(X)\times C(Y)\end{array} \right\}.
\end{eqnarray*}

Therefore,
\begin{equation}
\label{xiestrela2}
\Xi^*(\pi,g)=\left\{\begin{array}{lll}
\ \sup\{ g(\nu) : \nu \in \Pi(T)\},\ \ \ \mbox{if} \  \pi\in\Pi^{*}(T),\\
\\
+\infty, \ \mbox{in the other case}.
\end{array}\right.
\end{equation}

We know that the left size of (\ref{rockafeller}) is given by (\ref{leftsize}). By (\ref{tetaestrela2}) and (\ref{xiestrela2}), the right size of (\ref{rockafeller}) will be:

$$\sup_{(\pi,g) \in E^{*}}\left\{
-\left(\begin{array}{llll}
\left(\pi(c)+ \sup\{-g(\nu): \left\| \nu \right \| \leq 2 \}\right) + \sup\{g(\nu): \nu \in \Pi(T)\}, \\
 \, \, \, \, \, \, \mbox{if} \ \pi\in M^+(X\times Y)\cap \Pi^*(T)\\
\\
+\infty, \ \mbox{in the other case}
\end{array}\right)
\right\}$$

$$= \sup_{(\pi,g) \in E^{*}}\left\{
\left(\begin{array}{llll}
\left(-\pi(c)+ \inf\{g(\nu): \left\| \nu \right \| \leq 2 \}\right) - \sup\{g(\nu): \nu \in \Pi(T)\}, \\
 \, \, \, \, \, \, \mbox{if} \ \pi\in M^+(X\times Y)\cap \Pi^*(T)\\
\\
-\infty, \ \mbox{in the other case}
\end{array}\right)
\right\}$$

$$ = \sup \{-\pi(c), \ \pi\in M^+(X\times Y)\cap\Pi^*(T)\},$$
where the last equality is obtained taking $g=0$. \newline

We remark that if $\pi \in  M^+(X\times Y)\cap\Pi^*(T)$ then: \newline
$\pi(1) = 1$, (by definition of $\Pi^{*}(T)$) \newline
$\pi(u)\geq 0$ if $u \geq 0$, \newline
$\pi$ is linear in $C(X\times Y)$, \newline
Therefore, we have $\pi \in P(X\times Y)$, and, by definition of $\Pi^*(T)$, we get that $\pi \in \Pi(T)$.
The conclusion is that $M^+(X\times Y)\cap\Pi^*(T) = \Pi(T)$. So the right side of (\ref{rockafeller}) will be:
$$\sup \{-\pi(c), \ \pi\in \Pi(T)\}=-\inf\{\pi(c), \ \pi\in \Pi(T)\}.$$
Therefore, we conclude from (\ref{rockafeller}) that:

\[-\,\sup\left\{\alpha :\,\, \begin{array}{l} \exists (\varphi,\psi)\in C(X)\times C(Y), \\
 \alpha + \varphi(x)- \varphi(T_{1}(x,y))+\psi(y) -\psi(T_{2}(x,y))\leq c(x,y) \end{array}
\right\}\]
\[=-\inf\{\pi(c), \ \pi\in \Pi(T)\},\]
or, in another form that
\[\,\sup\left\{\alpha :\,\, \begin{array}{l} \exists (\varphi,\psi)\in C(X)\times C(Y), \\
 \alpha + \varphi(x)- \varphi(T_{1}(x,y))+\psi(y) -\psi(T_{2}(x,y))\leq c(x,y) \end{array}
\right\}\]
\[=\inf\{\pi(c), \ \pi\in \Pi(T)\}.\]

\end{proof}

\begin{proposition} Suppose $c$ is continuous. Denote $\displaystyle{\alpha=\inf_{\pi\in \Pi(T)} \int c(x,y) \, d\pi}$.
If there exist $\varphi \in C(X)$ and $\psi\in C(Y)$ satisfying
\begin{equation}\label{phipsi}
 \alpha +\varphi(x)-\varphi(T_{1}(x)) + \psi(y)-\psi(T_{2}(y)) \leq c(x,y) \ \ \forall (x,y)\in X \times Y,
 \end{equation}
then,
\[ \inf \{ c(x,y)+...+c(T^{n-1}(x,y)) - n \alpha : n\geq 1 \, \text{and}\,  (x,y)\in X\times Y \}>-\infty.\]

\end{proposition}

\begin{proof}
Suppose by contradiction that
\[ \inf \{ c(x,y)+...+c(T^{n-1}(x,y)) - n \alpha : n\geq 1 \, \text{and}\,  (x,y)\in X\times Y \}=-\infty.\]
Also suppose that there exists $\varphi$ and $\psi$ satisfying (\ref{phipsi}). Then we have
\[ \inf \{\varphi(x) - \varphi(T_{1}^{n}(x,y)) + \psi(y) - \psi(T_{2}^{n}(x,y)) : n\geq 0 \, \text{and} \, (x,y)\in X\times Y \}= -\infty.\]
This is impossible because $X$ and $Y$ are compact sets and $\varphi$ and $\psi$ are continuous functions.
\end{proof}

\begin{proposition}
 Suppose $X=Y=\{0,1\}^{\mathbb{N}}$, $T_{1}=T_{2}=\sigma$, and, that $c$ is a Lipschitz function, then there exists $\varphi(x)$  and $\psi(y)$ Lipschitz continuous such that
\[\alpha + \varphi(\sigma(x)) - \varphi(x) + \psi(\sigma(y)) - \psi(y) \leq c(x,y).\]
\end{proposition}

\begin{proof} Denote by $\beta$ a Lipschitz constant for $c$.

By definition of $\alpha$ there exist a increasing sequence $ \alpha_{n} \to \alpha $ and continuous functions $\varphi_{n}, \psi_{n}$ such that:
\[\alpha_{n} + \varphi_{n}(\sigma(x)) - \varphi_{n}(x) + \psi_{n}(\sigma(y)) - \psi_{n}(y) \leq c(x,y).\]
From this relation we have that if $\sigma^{m}(z) = x$ and $y_{0},...,y_{m-1}$ belongs to $Y$, then
\[\varphi_{n}(x) - \varphi_{n}(z) \leq \sum_{i=0}^{m-1} c(\sigma^{i}z,y_{i}) +\psi_{n}(y_{i}) - \psi_{n}(\sigma(y_{i})) -\alpha_{n}.\]
Therefore,
\[\inf \{ \sum_{i=0}^{m-1} c(\sigma^{i}z,y_{i}) +\psi_{n}(y_{i}) - \psi_{n}(\sigma(y_{i})) -\alpha_{n}\, : \,m\geq 0, \, \sigma^{m}(z) = x , \, y_{i} \in Y\} > -\infty.\]
Denote by $\overline{\varphi}_{n}(x)$ this infimum. Note that this function satisfies:
\[\alpha_{n} + \overline{\varphi}_{n}(\sigma(x)) - \overline{\varphi}_{n}(x) + \psi_{n}(\sigma(y)) - \psi_{n}(y) \leq c(x,y).\]
It is easy to see that  $\overline{\varphi}_{n}$ is Lipschitz continuous with the same Lipschitz constant $\beta$ of $c$.
Using this last inequality and  the same arguments used before (applied to $\psi$) we can construct a Lipschitz continuous function $\overline{\psi_{n}}$ with the same Lipschitz constant $\beta$ satisfying:
\[\alpha_{n} + \overline{\varphi}_{n}(\sigma(x)) - \overline{\varphi}_{n}(x) + \overline{\psi}_{n}(\sigma(y)) - \overline{\psi}_{n}(y) \leq c(x,y).\]
Note that we can add constants on $\overline{\varphi}_{n}$ and $\overline{\psi}_{n}$ and the conclusions are the same. Then, we get:
there exists Lipschitz continuous functions $\overline{\varphi}_{n}$ and $\overline{\psi}_{n}$ with Lipschitz constant $\beta$, which bounded  bellow and above, respectively, by $0$ and $\beta$, such that:
\[\alpha_{n} + \overline{\varphi}_{n}(\sigma(x)) - \overline{\varphi}_{n}(x) + \overline{\psi}_{n}(\sigma(y)) - \overline{\psi}_{n}(y) \leq c(x,y).\]
Now using Arzela-Ascoli theorem we obtain continuous functions $\varphi$ and $\psi$ satisfying:
\[\alpha + \varphi(\sigma(x)) - \varphi(x) + \psi(\sigma(y)) - \psi(y) \leq c(x,y).\]
Applying the same reasoning of the previous arguments we can also construct Lipschitz functions satisfying this inequality.
\end{proof}

\begin{proposition} Let $C(X,Y)$ be the set of continuous functions from $X\times Y$ to $\mathbb{R_{+}}$ with the uniform norm. The set of functions $c \in C(X,Y)$ with a unique Optimal Plan in $\Pi(T)$ is generic in $C(X,Y)$. The same is true for the Banach space $H(X,Y)$ of the Lipschitz functions with the usual norm.

\end{proposition}
\begin{proof}
The result follows from adapting  the proof of Theorem \ref{teogeneric}.
\end{proof}

\section{Zeta-measures for the second class of problems}  \label{zeta2}

In this section we suppose $X=Y=\{0,1\}^{\mathbb{N}}$ and $T_{1}=T_{2}=\sigma$ is the shift.
On this case we have $\Pi(T)=\Pi(\sigma)$ is the set of probabilities $\pi$ in $X\times Y$ such that project on $\sigma$-invariant measures in $X$ and $Y$.

\bigskip

 Bellow we consider the problem of finding a transport plan in $\Pi(\sigma)$ maximizing the integral of a cost $c$ strictly greater than zero. A transport plan maximizing this integral will be called a \textbf{maximizing plan}. By changing the signal of the cost we can get from this the analysis of usual
 minimization problem of Transport Theory.

\begin{definition}
For a fixed cost $c$ we define a probability measure in $X\times Y$ by the linear functional $\zeta_{\beta,n}: C(X\times Y) \to \mathbb{R}$, such that, to each $w \in C(X\times Y)$ we associate the number:
\[ \frac{\sum_{\nu,\mu \in Fix_{n}}e^{\beta.n.\int c(x,y) d\pi(\mu,\nu)}\int w \, d\pi(\mu,\nu)}{\sum_{\mu,\nu \in Fix_{n}}e^{\beta.n.\int c(x,y) d\pi(\mu,\nu)}},\]
where $Fix_{n}$ denote the set of invariant measures in $X=Y$ supported in a periodic orbit of length $n$, and, $\pi(\mu,\nu)$ denote a maximizing plan from $\mu$ to $\nu$ with cost function $c$ (we don't impose other conditions on the plan).
\end{definition}

This zeta-measures can be determined by Linear Algebra arguments. Indeed, note that  if $\mu, \nu \in Fix_{n}$, then the plan  $\pi(\mu,\nu)$ can be determined by the study of certain permutations (see page 5 in \cite{Vi1}).

\begin{theorem}
When $\beta, n$ goes to infinite, any limit measure $\pi_{\infty}$ of convergent subsequence of $\zeta_{\beta,n}$, in the weak* topology, is on $\Pi(\sigma)$. Also, if $c>0$, then, $\pi_{\infty}$ maximizes the integral of $c$ between the measures in $\Pi(\sigma)$.
\end{theorem}

\begin{proof}
We begin by proving that for $\beta,n$ fixed, the zeta-measure is in $\Pi(\sigma)$. Let $w$ be a function depending only on y. Then
\[ \zeta_{\beta,n}(w\circ\sigma)=\frac{\sum_{\mu,\nu \in Fix_{n}}e^{\beta.n.\int c(x,y) d\pi(\mu,\nu)}\int w\circ\sigma \, d\pi(\mu,\nu)}{\sum_{\mu,\nu \in Fix_{n}}e^{\beta.n.\int c(x,y) d\pi(\mu,\nu)}}\]
\[= \frac{\sum_{\mu,\nu \in Fix_{n}}e^{\beta.n.\int c(x,y) d\pi(\mu,\nu)}\int w\circ\sigma \, d\nu}{\sum_{\mu,\nu \in Fix_{n}}e^{\beta.n.\int c(x,y) d\pi(\mu,\nu)}}\]
\[=\frac{\sum_{\mu,\nu \in Fix_{n}}e^{\beta.n.\int c(x,y) d\pi(\mu,\nu)}\int w \, d\nu}{\sum_{\mu,\nu \in Fix_{n}}e^{\beta.n.\int c(x,y) d\pi(\mu,\nu)}}\]
\[=\frac{\sum_{\mu,\nu \in Fix_{n}}e^{\beta.n.\int c(x,y) d\pi(\mu,\nu)}\int w \, d\pi(\mu,\nu)}{\sum_{\mu,\nu \in Fix_{n}}e^{\beta.n.\int c(x,y) d\pi(\mu,\nu)}} = \zeta_{\beta,n}(w).\]
If $w$ depends only on $x$ the argument is similar. This shows that $\zeta_{\beta,n} \in \Pi(\sigma)$. Then, when $\beta, n$ goes to infinite, via a convergent subsequence, any limit measure $\pi_{\infty}$ of $\zeta_{\beta,n}$, in the weak* topology, will be  on $\Pi(\sigma)$.

\bigskip
Suppose $\zeta_{\beta_j,n_j}\to \pi_{\infty}$, when $j\to\infty$.

Let $\pi^* \in \Pi(\sigma)$ be a probability maximizing the integral of $c$. Let $\mu^*,\nu^*$, respectively,  the projection of $\pi^*$ in the first and second coordinates. Then, $\mu^*,\nu^*$ are invariant measures. Let $\mu_{n_{j}},\nu_{n_j}\in Fix_{n_j}$ subsequences converging to $\mu^*,\nu^*$ in the weak* topology. If $\pi_{n_j}$ is a maximizing plan from $\mu_{n_j}$ to $\nu_{n_j}$, then, there exist a subsequence $\pi_{n_i}$ converging to a maximizing plan $\pi$ from $\mu^*$ to $\nu^*$ (\cite{Vi2} page 77). Clearly
\[ \int c \, d\pi = \int c d \pi^* ,\]
and, therefore, $\pi$ is maximal. This means that $\pi$  maximizes the integral of $c$ among the measures in $\Pi(\sigma)$. We denote this integral by $I(c) $. We want to show that $\pi_{\infty}(c) \geq I(c)$. We note that the subsequence $\zeta_{\beta_i,n_i}$ converges to $\pi_{\infty}$ in the weak* topology.
From the above arguments we get:
 \bigskip
\begin{center}
\textsl{given $\varepsilon >0$, for sufficiently large $i$ there exists $\mu,\nu \in Fix_{n_i}$, such that,} $$\int c \, d\pi(\mu,\nu) > I(c) - \varepsilon.$$
\end{center}

Consider a fixed $\varepsilon >0$, such that, $(I(c)-\varepsilon)>0$, and, define:
\[A_{n_i}(\varepsilon) =\{ (\mu,\nu) \in Fix_{n_i} : \int c \, d\pi(\mu,\nu) \leq I(c) - \varepsilon\},\]
\[B_{n_i}(\varepsilon) = \{ (\mu,\nu) \in Fix_{n_i} : \int c \, d\pi(\mu,\nu) > I(c) - \varepsilon\}.\]

Then, we  have that
\[\sum_{(\mu,\nu) \in A_{n_i}(\varepsilon)}e^{\beta_i.n_i.\int c(x,y) d\pi(\mu,\nu)} \leq \sum_{(\mu,\nu) \in A_{n_i}(\varepsilon)}e^{\beta_i.n_i.(I(c) -\varepsilon)} \]
\[\leq e^{2n_i\log(2)+\beta_i.n_i.(I(c) -\varepsilon)},\]
and,
\[ \sum_{(\mu,\nu) \in A_{n_i}(\varepsilon)}e^{ \beta_i.n_i.\int c(x,y) d\pi(\mu,\nu)}\int c \, d\pi(\mu,\nu) \leq  e^{2n_i\log(2)+\beta_i.n_i.(I(c) -\varepsilon)}(I(c)-\varepsilon).\]
By other hand, if $n_i$ is sufficiently large, $B_{n_i}(\varepsilon/2)$ is not empty. Moreover,
\[ \sum_{(\mu,\nu) \in B_{n_i}(\varepsilon)}e^{\beta_i.n_i.\int c(x,y) d\pi(\mu,\nu)} \geq \sum_{(\mu,\nu) \in B_{n_i}(\varepsilon/2)}e^{\beta_i.n_i.\int c(x,y) d\pi(\mu,\nu)}\]
\[\geq e^{\beta_i.n_i.(I(c) -\varepsilon/2)},\]
and,
\[\sum_{(\mu,\nu) \in B_{n_i}(\varepsilon)}e^{\beta_i.n_i.\int c(x,y) d\pi(\mu,\nu)}\int c \, d\pi(\mu,\nu) \geq  e^{\beta_i.n_i.(I(c) -\varepsilon/2)}(I(c)-\varepsilon/2).\]
Then,
\[ 0\leq \lim_{i \to\infty} \frac{\sum_{(\mu,\nu) \in A_{n_i}(\varepsilon)}e^{\beta_i.n_i.\int c(x,y) d\pi(\mu,\nu)}}{\sum_{(\mu,\nu) \in B_{n_i}(\varepsilon)}e^{\beta_i.n_i.\int c(x,y) d\pi(\mu,\nu)}} \]
\[\leq \lim_{i \to\infty} \frac{e^{2n_i\log(2)+\beta_i.n_i.(I(c) -\varepsilon)}}{e^{\beta_i.n_i.(I(c) -\varepsilon/2)}} = \lim_{i \to\infty}e^{2n_i\log(2) - \beta_i.n_i.\varepsilon/2} =0,\]
and,
\[ 0 \leq \lim_{i \to\infty} \frac{\sum_{(\mu,\nu) \in A_{n_i}(\varepsilon)}e^{\beta_i.n_i.\int c(x,y) d\pi(\mu,\nu)}\int c \, d\pi(\mu,\nu)}{\sum_{(\mu,\nu) \in B_{n_i}(\varepsilon)}e^{\beta_i.n_i.\int c(x,y) d\pi(\mu,\nu)}\int c \, d\pi(\mu,\nu)}\]
\[ \leq \lim_{i \to\infty} \frac{e^{2n_i\log(2)+\beta_i.n_i.(I(c) -\varepsilon)}(I(c)-\varepsilon)}{e^{\beta_i.n_i.(I(c) -\varepsilon/2)}(I(c)-\varepsilon/2)}\]
\[ = \lim_{i \to\infty}e^{2n_i\log(2) - \beta_i.n_i.\varepsilon/2}\frac{I(c) -\varepsilon}{I(c)-\varepsilon/2} = 0.\]
Therefore,
\[\liminf_{i\to\infty} \frac{\sum_{(\mu,\nu) \in Fix_{n_i}}e^{\beta_i.n_i.\int c(x,y) d\pi(\mu,\nu)}\int c \, d\pi(\mu,\nu)}{\sum_{(\mu,\nu) \in Fix_{n_i}}e^{\beta_i.n_i.\int c(x,y) d\pi(\mu,\nu)}}\]
\[=\liminf_{i\to\infty} \frac{ \sum_{(\mu,\nu) \in B_{n_i}(\varepsilon)}e^{\beta_i.n_i.\int c(x,y) d\pi(\mu,\nu)}\int c \, d\pi(\mu,\nu)}{\sum_{(\mu,\nu) \in B_{n_i}(\varepsilon)}e^{\beta_i.n_i.\int c(x,y) d\pi(\mu,\nu)}} \]
\[\geq \liminf_{i\to\infty} \frac{ \sum_{(\mu,\nu) \in B_{n_i}(\varepsilon)}e^{\beta_i.n_i.\int c(x,y) d\pi(\mu,\nu)}(I(c)-\varepsilon)}{\sum_{(\mu,\nu) \in B_{n_i}(\varepsilon)}e^{\beta_i.n_i.\int c(x,y) d\pi(\mu,\nu)}} \]
\[\geq I(c)-\varepsilon.\]

Taking $\varepsilon \to 0$ we get
\[\liminf_{i\to\infty} \frac{\sum_{(\mu,\nu) \in Fix_{n_i}}e^{\beta_i.n_i.\int c(x,y) d\pi(\mu,\nu)}\int c \, d\pi(\mu,\nu)}{\sum_{(\mu,\nu) \in Fix_{n_i}}e^{\beta_i.n_i.\int c(x,y) d\pi(\mu,\nu)}} \geq I(c).\]
Then, using the fact that $\zeta_{\beta_i,n_i}$ converges to $\pi_{\infty}$, we finally get $\pi_{\infty}(c) \geq I(c)$.
\end{proof}


\bigskip

\begin{thebibliography}{99}


\bibitem{BG} R. Bissacot and E. Garibaldi, Weak KAM methods and ergodic optimal problems for countable Markov shifts, \emph{Bull. Braz. Math. Soc.} 41, N.3, 321-338, (2010).


\bibitem{Bousch1}
T. Bousch, Le poisson n'a pas d'ar\^etes, \emph{Annales de
l'Institut Henri Poincar\'e, Probabilit\'es et Statistiques},
\textbf{36}, 489-508, (2000).



\bibitem{Bousch-walters}
T. Bousch,
 La condition de Walters,
  \emph{Ann. Sci. ENS}, 34, pp.~287--311,  (2001).


\bibitem{BJ}
T. Bousch and O. Jenkinson, Cohomology classes of dynamically non-negative
$C^k$ functions, \emph{Inventiones Mathematiae}, 148 (2002), 207--217


\bibitem{BQ}
X. Bressaud  and A. Quas, Rate of approximation of minimizing measures,  \emph{Nonlinearity}, 20, no. 4, 845-853, (2007).

\bibitem{CLT}
G. Contreras, A. O. Lopes and Ph. Thieullen,  Lyapunov minimizing
measures for expanding maps of the circle, \emph{Ergodic Theory
and Dynamical Systems}, \textbf{21}, (2001), pp.~1379--1409.

\bibitem{CM}
D. Collier and I. D.  Morris, Approximating the maximum ergodic average via periodic orbits. \emph{Ergodic Theory Dynam. Systems}, 28, no. 4, 1081-1090,  (2008)

\bibitem{CLO}
G. Contreras, A. O. Lopes and E. R. Oliveira, Ergodic Transport Theory, periodic maximizing probabilities  and  the twist condition, preprint  (2011), Arxiv





\bibitem{Conze-Guivarch}
 J.-P. Conze \& Y. Guivarc,
 {\em Croissance des sommes ergodiques},
 manuscript, circa 1993.



\bibitem{GT1}
E. Garibaldi and Ph. Thieullen, Minimizing orbits in the discrete Aubry-Mather model, \emph{Nonlinearity} 24 (2011), no. 2, 563-611


\bibitem{GL2}
E. Garibaldi and  A. O. Lopes, On the Aubry-Mather Theory for  Symbolic Dynamics,
\emph{Erg Theo and Dyn Systems}, Vol 28 , Issue 3, 791-815 (2008)

\bibitem{Gi} N. Gigli,
Introduction to Optimal
Transport: Theory and
Applications, XXVIII Coloquio Brasileiro de Matematica, 2011, IMPA, Rio de Janeiro


\bibitem{GL4} E. Garibaldi and A. O. Lopes,
The effective potential and transshipment in Thermodynamic Formalism at temperature zero, to appear in Stoch. Dyn.


\bibitem{HY}
B. R. Hunt and G. C. Yuan,  Optimal orbits of hyperbolic systems.
\emph{Nonlinearity}  \textbf{12}, (1999), 1207-1224.



\bibitem{Jenkinson}
O. Jenkinson, Ergodic optimization, \emph{Discrete and Continuous Dynamical Systems, Series A}
\textbf{15} (2006), 197-224.


\bibitem{K1} B. Kloeckner,
Optimal transport and dynamics of expanding circle maps acting on measures, preprint (2010), Arxiv





\bibitem{Le} R. Leplaideur, A dynamical proof for the convergence of Gibbs measures at temperature zero. \emph{ Nonlinearity} \textbf{18}, no. 6, (2005), 2847-2880.



\bibitem{LOT}
A. O. Lopes, E. R. Oliveira and Ph. Thieullen,
The dual potential, the involution kernel and transport in ergodic optimization,
\emph{preprint}, (2008).


\bibitem{LOS} A. O. Lopes, E. R. Oliveira and D. Smania, Ergodic Transport Theory and
Piecewise Analytic Subactions for Analytic Dynamics. preprint (2011)


\bibitem{LO} A. O. Lopes and E. R.  Oliveira,
On the thin boundary of the fat attractor, preprint (2012)



\bibitem{LM1} A. O. Lopes and J. Mengue, Zeta measures  and Thermodynamic Formalism for temperature zero, \emph{Bulletin of the Brazilian Mathematical Society} 41 (3), 449-480, (2010)



\bibitem{Man} R. Mane, Ergodic Theory and Differentiable Dynamics, Springer Verlag (1987)



\bibitem{Mo}
I. D. Morris,  A sufficient condition for the subordination principle in ergodic optimization, \emph{ Bull. Lond. Math. Soc.} \textbf{39}, no. 2, (2007), 214-220.

\bibitem{PP}
W. Parry and M. Pollicott,
Zeta functions and the periodic orbit structure of hyperbolic dynamics,
\emph{Ast\'erisque} N. \textbf{187-188} (1990).



\bibitem{Ra}
S. T. Rachev and L. R\"uschendorf,
\emph{Mass transportation problems -- Volume I: Theory, Volume II: Applications},
Springer-Verlag, New York, 1998.


\bibitem{Ro} C. Robinson, Dynamical Systems, CRC Press, (1995)


\bibitem{Vi1}
C. Villani, \emph{Topics in optimal transportation}, AMS, Providence, (2003).

\bibitem{Vi2}
C. Villani, \emph{Optimal transport: old and new}, Springer-Verlag, Berlin, (2009).

\end{thebibliography}
\end{document}